\documentclass[paper=a4,english,fontsize=11pt,parskip=half,abstract=true]{scrartcl}
\usepackage{babel}
\usepackage[utf8]{inputenc}
\usepackage[T1]{fontenc}
\usepackage{amsmath,amssymb,amsthm}
\usepackage{mathtools}
\usepackage{thmtools}
\usepackage{microtype}
\usepackage{embedfile}
\usepackage[left=20mm,right=20mm,top=30mm,bottom=30mm]{geometry}
\usepackage[shortlabels]{enumitem}
\usepackage[
  pdfauthor={Benjamin Sambale},
  pdftitle={Universally extendible characters},
  pdfstartview={FitH}
]{hyperref}

\renewcommand{\phi}{\varphi}
\newcommand{\C}{\mathrm{C}}
\newcommand{\N}{\mathrm{N}}
\newcommand{\Z}{\mathrm{Z}}
\newcommand{\pcore}{\mathrm{O}}

\newcommand{\F}{\mathrm{F}}

\newcommand{\CC}{\mathbb{C}}
\newcommand{\QQ}{\mathbb{Q}}

\newcommand{\FF}{\mathbb{F}}

\newcommand{\Aut}{\mathrm{Aut}}
\newcommand{\Inn}{\mathrm{Inn}}
\newcommand{\Out}{\mathrm{Out}}
\newcommand{\foc}{\mathrm{foc}}
\newcommand{\GL}{\mathrm{GL}}

\newcommand{\SL}{\mathrm{SL}}
\newcommand{\AGL}{\operatorname{AGL}}
\newcommand{\ASL}{\operatorname{ASL}}

\newcommand{\Irr}{\operatorname{Irr}}
\newcommand{\Char}{\operatorname{Char}}
\newcommand{\IBr}{\operatorname{IBr}}

\newcommand{\Syl}{\operatorname{Syl}}

\newcommand{\Ker}{\operatorname{Ker}}

\newcommand{\diag}{\operatorname{diag}}

\newcommand{\rsub}{\mathop{\trianglerighteq\trianglerighteq}}
\newcommand{\sub}{\mathop{\unlhd\unlhd}}
\newcommand{\MF}{\mathfrak F}

\newtheorem{Thm}{Theorem}[section]

\newtheorem{Lem}[Thm]{Lemma}
\newtheorem{Prop}[Thm]{Proposition}
\newtheorem{Cor}[Thm]{Corollary}

\theoremstyle{definition}
\newtheorem{Def}[Thm]{Definition}
\newtheorem{Ex}[Thm]{Example}
\numberwithin{equation}{section}

\title{Universally extendible characters}
\author{Benjamin Sambale\footnote{Institut für Algebra, Zahlentheorie und Diskrete Mathematik, Leibniz Universität Hannover, Welfengarten 1, 30167 Hannover, Germany,
\href{mailto:sambale@math.uni-hannover.de}{sambale@math.uni-hannover.de}}
}
\date{\today}

\begin{document}
\embedfile[desc={GAP code example 3.2},
           mimetype={text/plain}]{example3.2.g}
\def\sectionautorefname{Section}
\frenchspacing
\maketitle
\begin{abstract}\noindent
Let $\theta$ be an ordinary character or a Brauer character of a finite group $N$. 
We show that $\theta$ extends to every overgroup $G=G_\theta\trianglerighteq N$ if and only if $\Z(N)\le\Ker(\theta)$ and the deflated character $\overline\theta$ of $\Inn(N)$ extends to $\Aut(N)_\theta$. This equivalence is compatible with fields of values. A similar investigation for extensions from subnormal subgroups leads to a new characteristic subgroup of $N$.  
\end{abstract}

\section{Introduction}\label{sec:statement}

All groups in this paper are finite.
An important problem in representation theory is to decide when a character $\theta$ of a normal subgroup $N\unlhd G$ extends to $G$, i.\,e. when there exists a character $\chi$ of $G$ with $\chi_N=\theta$. Necessarily, $\theta$ must be $G$-invariant. 
For irreducible $\theta$ the most prominent sufficient criteria make assumptions on the embedding $N\subseteq G$ and hold for ordinary characters and for Brauer characters alike:
\begin{itemize}
\item $G/N$ is cyclic (see \cite[Theorem~(8.12)]{Navarro}),
\item $N$ is a Hall subgroup of $G$, i.\,e. $\gcd(|N|,|G:N|)=1$ (see \cite[Theorem~(8.13)]{Navarro}),
\item $\theta$ extends to all preimages of Sylow subgroups of $G/N$ (see \cite[Theorem~(8.29)]{Navarro}),
\item $\gcd(o(\theta)\theta(1),|G:N|)=1$, where $o(\theta)$
denotes the determinantal order of $\theta$ (see \cite[Theorem~(8.23)]{Navarro}).
\end{itemize}
The aim of this paper is to derive a criterion independent of $G$, i.\,e. we characterize when $\theta$ extends to \emph{every} possible overgroup. 

Let $\Char_p(N)$ be the set of Brauer characters of $N$ with respect to an algebraically closed field of characteristic $p>0$. The irreducible Brauer characters are denoted by $\IBr_p(N)$ as usual.
For convenience, we denote the set of (irreducible) ordinary complex characters by $\Char_0(N):=\Char(N)$ and $\IBr_0(N):=\Irr(N)$ respectively. In this context, every element is regarded as $p$-regular when $p=0$. 
The kernel $\Ker(\theta)$ of $\theta\in\Char_p(N)$ is the intersection of the kernels of the irreducible constituents of $\theta$. 

\begin{Def}\label{def:ue}
Let $\theta\in\Char_p(N)$ with $p\ge 0$. If $N$ is a normal subgroup of a group $G$, we denote the stabilizer of $\theta$ in $G$ by $G_\theta$. We call $G$ an \emph{overgroup} of $\theta$ if $\theta$ is $G$-invariant, i.\,e. $G_\theta=G$ (for irreducible $\theta$, $(G,N,\theta)$ is then a (modular) character triple in the notation of \cite{Navarro}). 
Moreover, $\theta$ is called \emph{universally extendible} if $\theta$ extends to every overgroup, i.\,e. for every such $G$ there exists a $\chi\in\Char_p(G)$ with $\chi_N=\theta$.
Finally, $N$ itself is called \emph{universally extendible} if every ordinary character of $N$ is universally extendible.
\end{Def}

Recall that the inner automorphism group $\Inn(N)$ of $N$ is naturally isomorphic to $N/\Z(N)$. Every character of $N$ is a class function and, as such, stabilized by $\Inn(N)$. If $\Z(N)\le\Ker(\theta)$, $\theta$ \emph{deflates} to a character $\overline{\theta}$ of $N/\Z(N)\cong\Inn(N)$. 
Our main theorem incorporates fields of values. 

\begin{Thm}\label{thm:main}
Let $F\subseteq\CC$ be a field and $\theta\in\Char_p(N)$ with $p\ge 0$. Then $\theta$ has an extension with values in $F$ to every overgroup if and only if $\Z(N)\le\Ker(\theta)$ and $\overline\theta$ has an extension to $\Aut(N)_\theta$ with values in $F$.
\end{Thm}

\begin{Cor}\label{cor:ordinary}
A character $\theta$ of $N$ is universally extendible if and only if $\Z(N)\le\Ker(\theta)$ and $\overline\theta$ has an extension to $\Aut(N)_\theta$.
\end{Cor}
\begin{proof}
Take $F=\CC$ in \autoref{thm:main}.
\end{proof}

It is further possible to remove the $G$-invariance requirement on $\theta$. In fact, $\theta$ extends to every $G\trianglerighteq N$ if and only if $\Z(N)\le\Ker(\theta)$ and $\overline{\theta}$ extends to $\Aut(N)$. 

It may happen that $\theta$ is not universally extendible, but $2\theta$ is (\autoref{ex:dihedral}\eqref{ex:di3}). 
Recall that $\Char_p(N)$ can be identified with $\Char_0(N)$ whenever $p$ does not divide $|N|$ (see \cite[Theorem~(2.12)]{Navarro}). 
In this situation, $\theta$ can be universally extendible as a Brauer character, but not as an ordinary character when  $p\mid\lvert\Aut(N)\rvert$ (\autoref{ex:coprimecounter}).

For abelian groups $N$ only the multiples of the trivial character are universally extendible (cf. \autoref{cor:nilpotent}). On the opposite end we have:

\begin{Cor}\label{cor:grouplevel}
A group $N$ is universally extendible if and only if $\Z(N)=1$ and every $\chi\in\Irr(\Inn(N))$ extends to its stabilizer in $\Aut(N)$.
\end{Cor}

\autoref{cor:grouplevel} implies that complete groups $N$, like symmetric groups of degree $\ge 7$, are universally extendible (here $\Z(N)=1$ and $\Out(N)=1$). This is generalized in \autoref{cor:criteria} below. 
In \autoref{sec:schur} we study when extensions are afforded by a representation over a given field. 
Subsequently, we investigate extensions from subnormal subgroups in \autoref{sec:subnormal}. This leads to a new characteristic subgroup of $N$, which we compute in special cases in \autoref{sec:RN}.

\section{Proofs of the main theorems}\label{sec:proofs}

In the following we fix $\theta\in\Char_p(N)$ with $p\ge 0$. 
Set $A:=\Aut(N)_\theta$ and $\Out(N)_\theta:=A/\Inn(N)$. Let $\kappa\colon N\to\Inn(N)$, $x\mapsto\kappa_x$ be the conjugation map with $\kappa_x(y)=xyx^{-1}$ for $x,y\in N$. 
For $k\ge1$ let $\delta\colon N\to N^k$, $x\mapsto(x,\ldots,x)$, be the diagonal map. For a group $G$, let $\exp(G)$ denote the exponent of $G$, i.\,e. the smallest integer $e\ge 1$ such that $g^e=1$ for all $g\in G$. The symmetric and alternating groups of degree $d$ are denoted by $S_d$ and $A_d$ respectively.

The following two elementary observations will be used several times.

\begin{Lem}\label{lem:diagcommutator}
Let $H\le S_k$ be a transitive permutation group where $k$ is divisible by $\exp(N/N')$. Then $\delta(N)\le(N\wr H)'$.
\end{Lem}

\begin{proof}
For $x\in N$ and $1\le i\le k$ let $e_i(x)\in N^k$ be the element with $x$ in position $i$ and $1$ elsewhere, so that $\delta(x)=e_1(x)\cdots e_k(x)$. Set $W:=N\wr H$. For every $i$ there is some $h\in H$ with $he_1(x)h^{-1}=e_i(x)$ and therefore $e_1(x)^{-1}e_i(x)=e_1(x)^{-1}he_1(x)h^{-1}\in W'$. Since $e_1(x),\ldots,e_k(x)$ commute pairwise, multiplying over $i=1,\ldots,k$ yields 
\[e_1(x)^{-k}\delta(x)=\prod_{i=1}^ke_1(x)^{-1}e_i(x)\in W'.\]
Finally $x^k\in N'$ by the choice of $k$, so that $e_1(x)^k=e_1(x^k)\in e_1(N')\le W'$ and $\delta(x)\in W'$.
\end{proof}

\begin{Lem}\label{lem:wreath}
Let $W:=N\wr A_k$ where $k\ge5$. Let $\psi\in\IBr_p(W)$ with $\psi(1)<k$. Then $\psi_{N^k}=\psi(1)(\nu\times\ldots\times\nu)$ for some linear $\nu\in\IBr_p(N)$. In particular, $\psi(x,1,\ldots,1)=\psi(1)\nu(x)$ for every $p$-regular $x\in N$. If $k$ is divisible by $\exp(N/N')$, then $\psi(\delta(x))=\psi(1)$.
\end{Lem}

\begin{proof}
Since ordinary characters are Brauer characters with respect to some $p\nmid|N|$, we may assume that $p>0$. 
By \cite[Theorem~(8.21)]{Navarro} the irreducible Brauer characters of $N^k$ are the products $\nu_1\times\ldots\times\nu_k$ with $\nu_i\in\IBr_p(N)$. Let $\mu:=\nu_1\times\ldots\times\nu_k$ be a constituent of $\psi_{N^k}$. By Clifford's theorem \cite[Corollary~(8.7)]{Navarro}, $\psi_{N^k}$ is a multiple of the sum of the $A_k$-orbit of $\mu$, and therefore
\[|A_k:(A_k)_\mu|\mu(1)\le\psi(1)<k.\]
Every proper subgroup of $A_k$ has index at least $k$, since $A_k$ is simple and $|A_k|>|S_{k-1}|$. Consequently, $\mu$ is $A_k$-invariant. Since $A_k$ acts transitively on the $k$ factors, we have $\nu_1=\ldots=\nu_k=:\nu$. From $\nu(1)^k=\mu(1)<k\le2^k$ we deduce that $\nu$ is linear, so that $\mu(1)=1$ and $\psi_{N^k}=\psi(1)\mu$. For $p$-regular $x\in N$ this gives $\psi(x,1,\ldots,1)=\psi(1)\nu(x)$ and $\psi(\delta(x))=\psi(1)\nu(x)^k$. Finally, $\nu$ is trivial on $N'$, whence $\nu^k=1_N$ if $k$ is a multiple of $\exp(N/N')$.
\end{proof}

\begin{proof}[Proof of \autoref{thm:main}]
Suppose first that $\Z(N)\le\Ker(\theta)$ and $\overline{\theta}$ extends to $\chi\in\Char_p(A)$ with values in $F$. Let $\Delta$ be a representation of $A$ affording $\chi$. Let $G$ be an overgroup of $\theta$, and let $\gamma\colon G\to A$ be the conjugation map. Then $\Delta\circ\gamma$ is a representation of $G$ affording $\chi\circ\gamma\in\Char_p(G)$, which satisfies $\chi(\gamma(x))=\chi(\kappa_x)=\overline{\theta}(\kappa_x)=\theta(x)$ for $p$-regular $x\in N$. Hence, $\theta$ extends to $G$ with values in $F$. 

Assume conversely that $\theta$ has an extension with values in $F$ to every overgroup. Let $\theta=\sum_{i=1}^rm_i\theta_i$ with distinct $\theta_1,\ldots,\theta_r\in\IBr_p(N)$ and $m_i\ge1$, and set $m:=\max\{m_1,\ldots,m_r\}$. Choose an integer $k\ge\max(5,m+1)$ divisible by $\exp(N/N')$, and let $T:=A_k$ act naturally on $\Omega:=\{1,\ldots,k\}$. Let $W:=N\wr T=N^k\rtimes T$ be the corresponding wreath product. It is easy to check that $A$ acts on the left of $W$ via
\[^{\alpha}(x_1,\ldots,x_k,t):=(\alpha(x_1),\ldots,\alpha(x_k),t).\]
Set
\[G:=W\rtimes A\cong N^k\rtimes(T\times A)\]
and identify $N^k$, $A$, $T$ and $W$ with their canonical copies inside $G$. Let $\mu\colon N\to G$, $x\mapsto \delta(x^{-1})\kappa_x$. For $x,y\in N$ we compute
\[\mu(x)\mu(y)=\delta(x^{-1})\kappa_x\delta(y^{-1})\kappa_y=\delta(x^{-1})\delta(\kappa_x(y)^{-1})\kappa_x\kappa_y=\delta((xy)^{-1})\kappa_{xy}=\mu(xy),\]
i.\,e. $\mu$ is a homomorphism, and clearly injective. Moreover, $\alpha(\mu(x))=\mu(\alpha(x))$ for $\alpha\in A$. Hence, $A$ acts on $\mu(N)$ in the same way as it acts on $N$. Since $\delta(N)$ is centralized by $T$ and $T$ commutes with $A$, the subgroup $\mu(N)$ is centralized by $T$. For $x_1,\ldots,x_k,y\in N$ we compute
\[\mu(y)(x_1,\ldots,x_k)=(y^{-1}\kappa_y(x_1),\ldots,y^{-1}\kappa_y(x_k),\kappa_y)=(x_1,\ldots,x_k)\mu(y).\]
Hence, $\mu(N)$ is centralized by $W$. In particular, $\mu(N)\unlhd G$. Since $\mu$ commutes with the action of $A$, we may replace $G$ by an isomorphic copy in which $\mu$ becomes the inclusion map of $N$. Then $G$ is an overgroup of $\theta$, and $x=\delta(x^{-1})\kappa_x$, i.\,e. $\kappa_x=\delta(x)x$, for every $x\in N$.

By hypothesis, $\theta$ extends to some $\chi\in\Char_p(G)$ with values in $F$. Since $[N,W]=1$, the map $\pi\colon N\times W\to NW$, $(u,w)\mapsto uw$ is a surjective homomorphism, so that $\widetilde\chi:=\chi_{NW}\circ\pi$ lies in $\Char_p(N\times W)$. Thus, we can write
\[\widetilde\chi=\sum_{i=1}^r\sum_{\psi\in\IBr_p(W)}a_{i\psi}(\theta_i\times\psi)\]
with integers $a_{i\psi}\ge0$, where only the constituents $\theta_i$ of $\theta$ occur since $\widetilde\chi_{N\times1}=\theta$ and $(\theta_i\times\psi)_{N\times1}=\psi(1)\theta_i$. Comparing the restrictions to $N\times1$ and to $1\times W$ yields
\[m_i=\sum_\psi a_{i\psi}\psi(1),\qquad\chi_W=\sum_{i=1}^r\theta_i(1)\sum_\psi a_{i\psi}\psi.\]
Hence, every irreducible constituent $\psi$ of $\chi_W$ satisfies $\psi(1)\le m_i\le m<k$ for some $i$.
\autoref{lem:wreath} gives $\psi(\delta(x))=\psi(1)$ for $p$-regular $x\in N$. Summing over the irreducible constituents of $\chi_W$ yields $\chi(\delta(x))=\chi(1)$, i.\,e. $\delta(x)\in\Ker(\chi)$.

For $p$-regular $z\in\Z(N)$ we have $\kappa_z=1$ and therefore $z=\delta(z^{-1})\in\Ker(\chi)$. This shows $\theta(z)=\chi(z)=\chi(1)=\theta(1)$, i.\,e. $z\in\Ker(\theta)$. Since $\Z(N)$ is abelian, its $p$-elements (when $p>0$) form a normal $p$-subgroup of $N$ and therefore lie in $\pcore_p(N)\le\Ker(\theta)$ (see \cite[Lemma~(2.32)]{Navarro}). Hence, $\Z(N)\le\Ker(\theta)$.
More generally, $\delta(x)\in\Ker(\chi)$ implies $\chi(\kappa_x)=\chi(\delta(x)x)=\chi(x)=\theta(x)$ for $p$-regular $x\in N$. Thus, $\chi_A$ extends $\overline{\theta}$ on $\Inn(N)$ and takes only values of $\chi$. 
\end{proof}

\begin{Lem}\label{lem:orbit}
Let $G$ be an overgroup of $\theta\in\Char(N)$. If every irreducible constituent of $\theta$ extends to its stabilizer in $G$, then $\theta$ extends to $G$.
\end{Lem}

\begin{proof}
Since $\theta$ is $G$-invariant, its irreducible constituents form $G$-orbits. Let $\theta_1,\ldots,\theta_k\in\Irr(N)$ be representatives of those orbits. Let $m_i:=[\theta,\theta_i]$ be the multiplicity of $\theta_i$ in $\theta$. By hypothesis, each $\theta_i$ extends to some $\hat\theta_i\in\Irr(G_{\theta_i})$. Now $\chi:=\sum_{i=1}^km_i\hat\theta_i^G$ is an extension of $\theta$ because
\[\chi_N=\sum_{i=1}^km_i(\hat\theta_i^G)_N=\sum_{i=1}^km_i\sum_{G_{\theta_i}x\in G_{\theta_i}\backslash G}\theta_i^x=\sum_{\psi\in\Irr(N)}[\theta,\psi]\psi=\theta.\qedhere\]
\end{proof}

\begin{proof}[Proof of \autoref{cor:grouplevel}]
If $\Z(N)\ne1$, choose $\lambda\in\Irr(\Z(N))$ non-trivial and
$\theta\in\Irr(N)$ lying over $\lambda$. Then $\Z(N)\nsubseteq\Ker(\theta)$ and $\theta$ is not universally extendible by \autoref{cor:ordinary}. Now let $\Z(N)=1$, so that $N=\Inn(N)$ and $\overline\theta=\theta$. By \autoref{cor:ordinary}, an irreducible character of $N$ is universally extendible if and only if it extends to its stabilizer in $\Aut(N)$. Hence, the claim follows from \autoref{lem:orbit}.
\end{proof}

We record one more property of the overgroup $G$ built above.

\begin{Lem}\label{lem:NcapH}
There exists an overgroup $G$ satisfying $N\cap G'=\kappa^{-1}(A')$.
\end{Lem}

\begin{proof}
Let $G$ be the overgroup constructed in the proof of \autoref{thm:main} (observe that the construction does not require universal extendibility).
Recall from the proof that $x=\delta(x^{-1})\kappa_x$ for $x\in N$. Since $W=N^k\rtimes T$ with $T$ transitive on $\Omega$ and $k$ is divisible by $\exp(N/N')$, \autoref{lem:diagcommutator} gives $\delta(x)\in W'\le G'$ for every $x\in N$. Hence, $x\in G'$ if and only if $\kappa_x\in G'$. Since $G/W\cong A$, we have $\kappa_x\in G'\le WA'$ if and only if $\kappa_x\in A'$. 
\end{proof}

\section{Consequences}\label{sec:consequences}

\begin{Prop}\label{prop:coprime}
Let $p$ be a prime not dividing $|N|$ and let $\theta\in\Irr(N)$. Then $\theta$ is universally extendible as an ordinary character if and only if $\theta$ is universally extendible as a $p$-Brauer character.
\end{Prop}

\begin{proof}
By \autoref{thm:main} we may assume that $\Z(N)\le\Ker(\theta)$, since otherwise $\theta$ is universally extendible in neither sense. If $\theta$ has an ordinary extension to every overgroup, then the restrictions of these extensions to the $p$-regular classes are $p$-Brauer extensions by \cite[Corollary~(2.9)]{Navarro}. 

Conversely, let $\overline\theta$ extend to some $\chi\in\Char_p(A)$. Let $q$ be a prime and let $P$ be the preimage in $A$ of a Sylow $q$-subgroup of $\Out(N)_\theta$. If $q\ne p$, then $|P|=|\Inn(N)||P:\Inn(N)|$ is not divisible by $p$. Hence, $\Char_p(P)=\Char(P)$, and $\chi_P$ is an ordinary extension of $\overline\theta$. If $q=p$, then $p$ divides neither $\overline\theta(1)$ nor the determinantal order $o(\overline\theta)$, because both divide $|\Inn(N)|$. Since $|P:\Inn(N)|$ is a power of $p$, the character $\overline\theta$ extends to $P$ by \cite[Corollary~8.16]{Isaacs}. Now $\overline\theta$ extends to $A$ by \cite[Corollary~11.31]{Isaacs} and the claim follows from \autoref{cor:ordinary}.
\end{proof}

\begin{Ex}\label{ex:coprimecounter}
\autoref{prop:coprime} fails for reducible characters: there exists a group $N$ of order $2\cdot 3^{10}$ with a unique non-trivial linear character $\theta$ such that $2\theta$ is universally extendible as a $5$-Brauer character, but not as an ordinary character (see the GAP~\cite{GAP4161} code attached to this pdf).
\end{Ex}

We present some useful sufficient criteria for universal extendibility.

\begin{Prop}\label{cor:criteria}\hfill
\begin{enumerate}[(a)]
\item\label{cor:brauercrit} Let $\theta\in\IBr_p(N)$ with $\Z(N)\le\Ker(\theta)$. If the Schur multiplier $H^2(\Out(N)_\theta,\CC^\times)$ is a $p$-group (trivial if $p=0$), then $\theta$ is universally extendible.
 
\item If $\Z(N)=1$ and every Sylow subgroup of $\Out(N)$ is cyclic, then $N$ is universally extendible.
\end{enumerate}
\end{Prop}

\begin{proof}\hfill
\begin{enumerate}[(a)]
\item Let $B:=\Out(N)_\theta$. Suppose first that $p>0$. 
Let $E$ be an algebraically closed field of characteristic $p$. By \cite[Theorem~(8.15)]{Navarro} the triple
$(A,\Inn(N),\overline\theta)$ has an obstruction class in $H^2(B,E^\times)$ which
is trivial if and only if $\overline\theta$ extends to $A$. By
\cite[Proposition~2.1.14]{Karpilovsky}, $H^2(B,E^\times)$ is isomorphic to the
$p'$-part of the Schur multiplier $H^2(B,\CC^\times)$ and therefore trivial.
Hence, $\overline\theta$ extends to $A$. If $p=0$ and $H^2(B,\CC^\times)=1$, then $\overline\theta$ extends to $A$ by \cite[Theorem~11.7]{Isaacs}. In both cases \autoref{thm:main} applies.

\item Let $\theta\in\Irr(N)$. Every Sylow subgroup of $\Out(N)_\theta$ is contained in a Sylow
subgroup of $\Out(N)$ and is therefore cyclic. This implies that $H^2(\Out(N)_\theta,\CC^\times)=1$ by \cite[Corollary~11.21]{Isaacs}, so $\theta$ is universally extendible by \ref{cor:brauercrit}. Now the claim follows from \autoref{cor:grouplevel} (cf. \cite[Corollaries~11.22 and 11.31]{Isaacs}).\qedhere
\end{enumerate}
\end{proof}

By \cite[Problem~(2.7)]{Navarro}, the linear Brauer characters of $N$ can be identified with $\Irr\bigl(N/N'\pcore^{p'}(N)\bigr)$, i.\,e. with the ordinary linear characters of $p'$-order (when $p>0$). 

\begin{Prop}\label{cor:linearcrit}
For every linear character $\theta\in\IBr_p(N)$ the following are equivalent:
\begin{enumerate}[(1)]
\item\label{item:lin1} $\theta$ is universally extendible.
\item\label{item:lin3} $N\cap G'\le\Ker(\theta)$ for every overgroup $G$ of $\theta$.
\item\label{item:lin2} $\kappa^{-1}(A')\le\Ker(\theta)$.
\end{enumerate}
\end{Prop}

\begin{proof}\hfill
\begin{enumerate}[leftmargin=20pt,align=left,labelwidth=10pt]
\item[\ref{item:lin1}$\Rightarrow$\ref{item:lin3}:] An extension $\hat\theta$ of $\theta$ to an overgroup $G$ has degree
$1$. Hence, $G'\le\Ker(\hat\theta)$ and $N\cap G'\le\Ker(\theta)$. 

\item[\ref{item:lin3}$\Rightarrow$\ref{item:lin2}:] Apply \ref{item:lin3} to the overgroup $G$ of
\autoref{lem:NcapH}.

\item[\ref{item:lin2}$\Rightarrow$\ref{item:lin1}:] By hypothesis, $\Z(N)=\Ker(\kappa)\le\kappa^{-1}(A')\le\Ker(\theta)$ and $\Inn(N)\cap A'\le\Ker(\overline{\theta})$. By deflation, we may regard $\overline{\theta}$ as a linear character of $\Inn(N)/(\Inn(N)\cap A')\cong \Inn(N)A'/A'$. As such, it extends to the abelian group $A/A'$. By inflation, $\overline{\theta}$ extends to $A$. Hence, \ref{item:lin1} follows from \autoref{cor:ordinary}.\qedhere
\end{enumerate}
\end{proof}

We conclude with some examples.

\begin{Ex}\label{ex:dihedral}\hfill
\begin{enumerate}[(i)]
\item Let
\[N=\langle a,b\mid a^{12}=1,\ b^2=a^6,\ a^b=a^{-1}\rangle=\mathtt{SmallGroup}(24,4)\]
be the dicyclic group of order $24$. Then $\Z(N)=\langle a^6\rangle$ has order $2$, and there is a unique $\theta\in\Irr(N)$ of degree $2$ with $\Ker(\theta)=\Z(N)$, so that $\overline\theta\in\Irr(\Inn(N))$ is faithful. Since $\Z(N)$ is characteristic in $N$, we have $A=\Aut(N)\cong D_8\times S_3$, but $\overline\theta$ does not extend to $A$. Consequently, $\theta$ is not universally extendible. Let $\theta^0$ be the restriction of $\theta$ to the $2$-regular elements. Since $\IBr_2(N)=\{1_N,\theta^0\}$, we have $\Out(N)_{\theta^0}=\Out(N)\cong C_2^2$, whose Schur multiplier $C_2$ has trivial $2'$-part. Hence, $\theta^0$ is universally extendible by \autoref{cor:criteria}\ref{cor:brauercrit}.

\item Let $\theta$ be the unique non-trivial linear character of $N\cong D_{10}$. Since $A=\Aut(N)\cong C_5\rtimes C_4$, we have $\kappa^{-1}(A')=N'\le\Ker(\theta)$, so $\theta$ is universally extendible. Note that $\theta$ is rational, but neither of the two extensions to $A$ is rational (their field of values is $\QQ(i)$). 

\item\label{ex:di3} Let $\theta$ be the unique non-trivial linear character of the generalized dihedral group $N=C_3^2\rtimes C_2$. Here
\[A=\Aut(N)\cong\AGL(2,3)\cong C_3^2\rtimes\GL(2,3)\]
is the affine linear group and $A'=\ASL(2,3)$. The inverting automorphism is the central involution of $\SL(2,3)$, so $\Inn(N)\le A'$ and therefore $\kappa^{-1}(A')=N\nsubseteq\Ker(\theta)$. By \autoref{cor:linearcrit}, $\theta$ is not universally extendible. 
On the other hand, the inflation to $A$ of any faithful $\chi\in\Irr(\GL(2,3))$ of degree $2$ is an extension of $2\theta$. Hence, $2\theta$ is universally extendible by \autoref{cor:ordinary}. Finally, let $\lambda\in\Irr(N)$ be of degree $2$. Then $\Out(N)_{\theta+\lambda}=\Out(N)_\lambda\cong S_3$ has trivial Schur multiplier. By \cite[Theorem~11.7]{Isaacs}, $\overline\theta$ and $\overline\lambda$ both extend to $\Aut(N)_{\theta+\lambda}$. Therefore, $\theta+\lambda$ is universally extendible by \autoref{cor:ordinary}.
\end{enumerate}
\end{Ex}

\section{Realizable characters over given fields}\label{sec:schur}

In this section we investigate when an extension $\chi\in\Char(G)$ of an ordinary character $\theta$ of $N$ is \emph{realizable} over a field $F\subseteq\CC$, i.\,e. there exists an $F$-representation of $G$ affording $\chi$. 

\begin{Thm}\label{thm:schurmain}
Let $\theta\in\Char(N)$ and let $F\subseteq\CC$ be a field. Then $\theta$ has an extension realizable over $F$ to every overgroup if and only if $\Z(N)\le\Ker(\theta)$ and $\overline\theta$ has an extension to $\Aut(N)_\theta$ realizable over $F$.
\end{Thm}

\begin{proof}
Both implications in the proof of \autoref{thm:main} produce the desired representations.
\end{proof}

Recall that an irreducible character $\psi$ is realizable over $F$ if and only if $F(\psi)=F$ and the Schur index $m_F(\psi)$ equals $1$ (see \cite[Definition~10.1 and Corollary~10.2(d)]{Isaacs}). 

\begin{Lem}\label{lem:schurext}
Let $H\le G$ and let $\psi\in\Irr(H)$ be realizable over a field $F\subseteq\CC$. Then an extension $\chi$ of $\psi$ to $G$ is realizable over $F$ if and only if $F(\chi)=F$.
\end{Lem}

\begin{proof}
One direction is trivial. So let $F(\chi)=F$. Since $\chi_H=\psi$ is irreducible, so is $\chi$. Moreover, $F(\psi)\subseteq F(\chi)=F$ and $[\psi^G,\chi]=[\psi,\chi_H]=1$ by Frobenius reciprocity. Now \cite[Lemma~10.4]{Isaacs} gives $m_F(\chi)\mid m_F(\psi)=1$. Hence $\chi$ is realizable over $F$.
\end{proof}

By \autoref{lem:schurext}, the conditions in \autoref{thm:main} and \autoref{thm:schurmain} are equivalent for every $\theta\in\Irr(N)$ realizable over $F$.

\begin{Ex}\label{ex:dic}\hfill
\begin{enumerate}[(i)]
\item Let $N\cong C_{15}\rtimes C_4=\mathtt{SmallGroup}(60,7)$ with $\Z(N)=1$ and $N'\cong C_{15}$. Let $\theta\in\Irr(N)$ be the unique character of degree $2$ with $\Ker(\theta)\cong C_5$, that is, the inflation of the faithful character of the dicyclic group $N/C_5$ of order $12$. Then $\theta$ has rational values and Frobenius--Schur indicator $-1$. Thus, $m_\QQ(\theta)=2$ and $\theta$ is not realizable over $\QQ$. One can show that $\theta$ has an extension to $\Aut(N)$ with values in $\QQ(i)$ and Schur index $1$. Hence, $\theta$ has an extension realizable over $\QQ(i)$ to every overgroup by \autoref{thm:schurmain}.

\item Let $V:=\FF_{81}$, regarded as a vector space of dimension $4$ over $\FF_3$. Let $a\in\GL(V)$ be the multiplication by a generator of $\FF_{81}^\times$. Let $b\in\GL(V)$ be the Frobenius map $x\mapsto x^3$. We set 
\[H_0:=\langle a^4,a^2b^2\rangle,\qquad H:=\langle a^2,b^2\rangle,\qquad N:=V\rtimes H,\qquad K:=V\rtimes H_0.\]
Then $H_0\cong C_5\rtimes C_8$, $H\cong D_{10}\times C_8$ and $|N|=6480$. Moreover, $\Z(N)=1$. Let $\theta\in\Irr(N)$ be the unique linear character with $\Ker(\theta)=K$. It takes the values $\pm1$ and is realizable over $\QQ$. Here $A=\Aut(N)_\theta$ has order $25920$ with $A/K\cong Q_8$ and $\Z(A/K)=N/K$. In particular, $2\overline{\theta}$ has a unique (rational) extension $\chi\in\Char(A)$. By \autoref{thm:main}, $2\theta$ has rational extensions to every overgroup.
However, by \autoref{thm:schurmain}, there exists an overgroup to which $2\theta$ has no extension realizable over $\QQ$, since $\chi$ is not realizable over $\QQ$. Therefore, the remark after \autoref{lem:schurext} fails for reducible characters. 
\end{enumerate}
\end{Ex}

\section{Subnormal subgroups}\label{sec:subnormal}

Throughout this section $\theta\in\Char_p(N)$ with $p\ge 0$. Recall that $N$ is \emph{subnormal} in a group $G$ if there exists a chain $N=N_1\unlhd\ldots\unlhd N_k=G$. In this case we write $N\sub G$. Since normal subgroups are subnormal, one may ask which characters extend to every group $G$ containing $N$ as a subnormal subgroup. For non-normal $N$ the natural analog of $G$-invariance is invariance under $G$-\emph{fusion}, i.\,e. $\theta(x^g)=\theta(x)$ whenever $x,x^g\in N$ are $p$-regular and $g\in G$. We call $\theta$ \emph{subnormally extendible} if $\theta$ extends to every group $G\rsub N$ in which $\theta$ is fusion invariant. If $N\unlhd G$, then $G$-fusion invariance is just $G$-invariance. Clearly, every subnormally extendible character is universally extendible. We will see in \autoref{thm:submain} that the converse is false.

\begin{Thm}\label{thm:sublinear}
Let $\theta$ be subnormally extendible and let $N\sub G$ such that $\theta$ is $G$-fusion invariant. Then $\theta$ is the restriction of a sum of linear characters of $G$. In particular, every irreducible constituent of $\theta$ is linear and extends to $G$. Moreover, $N\cap G'\le\Ker(\theta)$.
\end{Thm}

\begin{proof}
Choose an integer $k\ge\max(5,\theta(1)+1)$ and let $H:=G\wr A_k=G^k\rtimes A_k$, where $A_k$ permutes the $k$ factors naturally. We identify $G$ with $G\times1\times\ldots\times1$ and $N$ with $N\times1\times\ldots\times1$. A subnormal chain from $N$ to $G$ yields $N\sub G\unlhd G^k\unlhd H$, so $N$ is subnormal in $H$.

Let $x\in N$ and $h=(g_1,\ldots,g_k)t\in H$ with $g_1,\ldots,g_k\in G$ and $t\in A_k$. Conjugation by $t$ moves the first component of $G^k$ to position $t(1)$, and conjugation by $(g_1,\ldots,g_k)$ acts componentwise. Hence $hxh^{-1}\in N$ if and only if $t(1)=1$ and $g_1xg_1^{-1}\in N$, and in this case $hxh^{-1}=g_1xg_1^{-1}$. Thus $H$-fusion of elements of $N$ is $G$-fusion, and $\theta$ is $H$-fusion invariant. By hypothesis there exists $\chi\in\Char_p(H)$ with $\chi_N=\theta$.

Write $\chi=\sum_\psi a_\psi\psi$ with $\psi\in\IBr_p(H)$ and integers $a_\psi\ge0$. Every $\psi$ with $a_\psi>0$ satisfies $\psi(1)\le\chi(1)=\theta(1)<k$, so \autoref{lem:wreath}, applied to $G$ in place of $N$, provides a linear $\nu_\psi\in\IBr_p(G)$ with $\psi(g)=\psi(1)\nu_\psi(g)$ for every $p$-regular $g\in G$. Hence
\[\chi_G=\sum_\psi a_\psi\psi(1)\nu_\psi\]
is a sum of linear characters in $\IBr_p(G)$, and $\theta=\chi_N=\sum_\psi a_\psi\psi(1)(\nu_\psi)_N$ is a sum of the linear characters $(\nu_\psi)_N\in\IBr_p(N)$. Since $\IBr_p(N)$ is linearly independent, every irreducible constituent of $\theta$ has the form $(\nu_\psi)_N$ and therefore extends to the linear character $\nu_\psi$ of $G$. As $G'\le\Ker(\nu_\psi)$, we obtain $N\cap G'\le\Ker(\theta)$.
\end{proof}

A linear character $\theta\in\IBr_p(N)$ is $G$-fusion invariant if and only if the \emph{focal subgroup}
\[\foc_G(N):=\langle x^{-1}x^g:x,x^g\in N,\,g\in G\rangle\]
lies in $\Ker(\theta)$. Note that $N'\le\foc_G(N)\le N\cap G'$. 

We will show next that fusion invariance is dispensable.

\begin{Lem}\label{lem:killfusion}
Let $N\sub G$. Then there exists a group $H\rsub N$ such that $\foc_H(N)=N'$ and $N\cap G'\le N\cap H'$.
\end{Lem}

\begin{proof}
Let $U:=N/N'$, $k:=\exp(U)$ and $\pi\colon N\to U$ be the natural map. Let $V:=U^k$ and $W:=V\rtimes\langle t\rangle\cong U\wr C_k$. Set $H:=W\times G$ and
\[f\colon N\to H,\qquad x\mapsto\bigl(\delta(\pi(x)),x\bigr).\]
Then $f$ is an injective homomorphism, and $\pi(yxy^{-1})=\pi(x)$ gives
\[(v,y)f(x)(v,y)^{-1}=\bigl(\delta(\pi(x)),yxy^{-1}\bigr)=f(yxy^{-1})\]
for $x,y\in N$ and $v\in V$. Hence $f(N)\unlhd V\times N$, and a subnormal chain from $N$ to $G$ yields $V\times N\sub V\times G\unlhd H$, so that $f(N)$ is subnormal in $H$. We identify $N$ with $f(N)$.

Since $V$ is abelian and $t$ fixes $\delta(U)$ pointwise, $\delta(U)$ lies in the center of $W$. For $x\in N$ and $h=(w,g)\in H$ we therefore have $hf(x)h^{-1}=\bigl(\delta(\pi(x)),gxg^{-1}\bigr)$, and this element lies in $f(N)$ only if $gxg^{-1}\in N$ and $\pi(gxg^{-1})=\pi(x)$, i.\,e. $x^{-1}gxg^{-1}\in N'$. Thus $\foc_H(N)=N'$. As $\langle t\rangle$ permutes the coordinates of $V$ transitively, \autoref{lem:diagcommutator} yields $\delta(U)\times1\le W'\times1\le H'$, and $1\times G'\le H'$. Hence every $x\in N\cap G'$ satisfies $f(x)=\bigl(\delta(\pi(x)),1\bigr)(1,x)\in H'$.
\end{proof}

Subnormal extendibility turns out to be governed by a single characteristic subgroup of $N$. We write
\[R(N):=\bigl\langle N\cap G':N\sub G\bigr\rangle\le N.\]
Let $N\sub G$ and $\omega\colon N\to G$ be the inclusion map. For $\alpha\in\Aut(N)$, $\omega\circ\alpha^{-1}$ is another inclusion of $N$ as a subnormal subgroup of $G$. Since 
\[(\omega\alpha^{-1})^{-1}(N\cap G')=\alpha(\omega^{-1}(N\cap G'))=\alpha(N\cap G'),\]
$R(N)$ is invariant under $\alpha$. Hence $R(N)$ is characteristic in $N$. Taking $G=N$ gives $N'\le R(N)$, so that $N/R(N)$ is abelian.

The following main theorem is the analog of \autoref{cor:linearcrit}. 

\begin{Thm}\label{thm:submain}
For $\theta\in\Char_p(N)$ the following assertions are equivalent:
\begin{enumerate}[(1)]
\item\label{item:sm1} $\theta$ is subnormally extendible.
\item\label{item:sm2} $\theta$ extends to every group $G\rsub N$. 
\item\label{item:sm3} Every irreducible constituent of $\theta$ is subnormally extendible.
\item\label{item:sm4} $R(N)\le\Ker(\theta)$.
\end{enumerate}
\end{Thm}

\begin{proof}\hfill
\begin{enumerate}[leftmargin=20pt,align=left,labelwidth=10pt]
\item[\ref{item:sm1}$\Rightarrow$\ref{item:sm4}:] Let $\theta$ be subnormally extendible and let $N\sub G$. Since $\theta$ is invariant under $N$-fusion, \autoref{thm:sublinear} with $G=N$ gives $N'\le\Ker(\theta)$. Let $H$ be as in \autoref{lem:killfusion}. Since $\foc_H(N)=N'\le\Ker(\theta)$, the character $\theta$ is $H$-fusion invariant, and \autoref{thm:sublinear} yields $N\cap G'\le N\cap H'\le\Ker(\theta)$. As $G$ was arbitrary, $R(N)\le\Ker(\theta)$.

\item[\ref{item:sm4}$\Rightarrow$\ref{item:sm2}:] Let $R(N)\le\Ker(\theta)$ and let $N\sub G$. Since $N'\le R(N)\le\Ker(\theta)$, every irreducible constituent $\lambda$ of $\theta$ is linear, and $N\cap G'\le R(N)\le\Ker(\lambda)$. 
The deflation of $\lambda$ to $N/(N\cap G')\cong NG'/G'$ extends to the abelian group $G/G'$.
Hence, $\lambda$ extends to a linear character $\hat\lambda\in\IBr_p(G)$. If $\theta=\sum_\lambda m_\lambda\lambda$, then $\sum_\lambda m_\lambda\hat\lambda$ is a sum of linear characters of $G$ extending $\theta$.

\item[\ref{item:sm2}$\Rightarrow$\ref{item:sm1}:] This is clear.

\item[\ref{item:sm3}$\Leftrightarrow$\ref{item:sm4}:] Since $\Ker(\theta)$ is the intersection of the kernels of the irreducible constituents of $\theta$, condition \ref{item:sm4} holds for $\theta$ if and only if it holds for every irreducible constituent of $\theta$. Now apply the equivalence \ref{item:sm1}$\Leftrightarrow$\ref{item:sm4} to the constituents.\qedhere
\end{enumerate}
\end{proof}

In contrast to \autoref{ex:coprimecounter}, the subnormal extendibility does not depend on the prime by \autoref{thm:submain}.

It is natural to ask when $\theta$ extends to every group $G$ containing $N$ as an arbitrary subgroup in which $\theta$ is fusion invariant.

\begin{Prop}\label{prop:arbitrary}
Suppose that $\theta$ extends to every group $G$ which contains $N$ as a subgroup and in which $\theta$ is $G$-fusion invariant. Then $\theta=\theta(1)1_N$.
\end{Prop}

\begin{proof}
Choose an integer $k\ge\max(5,\theta(1)+1)$ divisible by $\exp(N/N')$ and let $W:=N\wr T=N^k\rtimes T$, where $T:=A_k$ permutes the $k$ factors naturally. We identify $N$ with $\delta(N)\le W$.

Let $x,y\in N$ and $w=(y_1,\ldots,y_k)t\in W$ with $y_1,\ldots,y_k\in N$ and $t\in T$ such that $\delta(x)^w=\delta(y)$. The components of $\delta(x)^w$ are $x^{y_1},\ldots,x^{y_k}$ in some order, so $y$ is conjugate to $x$ in $N$ and $\theta(y)=\theta(x)$. Hence, every character in $\Char_p(N)$ is $W$-fusion invariant, and by hypothesis there is some $\chi\in\Char_p(W)$ with $\chi_N=\theta$.

Every irreducible constituent $\psi$ of $\chi$ satisfies $\psi(1)\le\chi(1)=\theta(1)<k$, so that $\psi(\delta(x))=\psi(1)$ for $p$-regular $x\in N$ by \autoref{lem:wreath}. Summing over the irreducible constituents of $\chi$ yields $\theta(x)=\chi(\delta(x))=\chi(1)=\theta(1)$ for every $p$-regular $x\in N$.
\end{proof}

In the next section we will show that non-trivial subnormally extendible characters actually exist.

\section{Computing \texorpdfstring{$R(N)$}{R(N)}}\label{sec:RN}

For a given $N$, it is not obvious how to compute $R(N)$. In this section we derive a lower bound and compute $R(N)$ in special situations. Let $1=\Z_0(N)\le\Z_1(N)\le\ldots$ be the upper central series of $N$ and let $\Z_\infty(N)$ be its last term, the hypercenter of $N$. 

\begin{Lem}\label{lem:diagsubnormal}
For $k\ge1$ and $i\ge0$ we have $\delta(N)\Z_i(N)^k\unlhd \delta(N)\Z_{i+1}(N)^k$. In particular, $\delta(N)$ is subnormal in
$\delta(N)\Z_\infty(N)^k$.
\end{Lem}

\begin{proof}
Let $\overline N:=N/\Z_i(N)$ and let $\pi\colon N^k\to\overline N^k$ be induced by the natural map $N\to\overline N$. Since $\Z_{i+1}(N)/\Z_i(N)=\Z(\overline N)$, we have $\delta(N)\Z_i(N)^k=\pi^{-1}(\delta(\overline N))$ and $\delta(N)\Z_{i+1}(N)^k=\pi^{-1}(\delta(\overline N)\Z(\overline N)^k)$. As $\Z(\overline N)^k$ centralizes $\delta(\overline N)$, the claim follows by taking preimages.
\end{proof}

\begin{Thm}\label{thm:subnecessary}
We have $\Z_\infty(N)\kappa^{-1}(\Aut(N)')\le R(N)$. 
\end{Thm}

\begin{proof}
Set $Z:=\Z_\infty(N)$, $k:=\exp(Z/Z')$ and $W:=N\wr C_k=N^k\rtimes\langle t\rangle$, where $t(x_1,\ldots,x_k)t^{-1}=(x_k,x_1,\ldots,x_{k-1})$. Then $L:=\delta(N)Z^k$ is normalized by $t$. Define 
\[G:=L\langle t\rangle\le W.\] 
By \autoref{lem:diagsubnormal}, $\delta(N)$ is subnormal in $L\unlhd G$. Moreover $Z^k\rtimes\langle t\rangle\le G$ and $\langle t\rangle$ permutes the $k$ coordinates of $Z^k$ transitively, so \autoref{lem:diagcommutator} yields $\delta(Z)\le(Z^k\rtimes\langle t\rangle)'\le G'$. Identifying $N$ with $\delta(N)$ we obtain $Z\le N\cap G'\le R(N)$.

Applied to the trivial character of $N$, where $A=\Aut(N)$, \autoref{lem:NcapH} provides a group $G\unrhd N$ with $N\cap G'=\kappa^{-1}(\Aut(N)')$, so that also $\kappa^{-1}(\Aut(N)')\le R(N)$.
\end{proof}

\begin{Cor}\label{cor:nilpotent}
If $N$ is nilpotent, then $R(N)=N$. In particular, only the multiples of the trivial character are subnormally extendible.
\end{Cor}

\begin{proof}
Since $\Z_\infty(N)=N$, \autoref{thm:subnecessary} gives $R(N)=N$. Hence, the claim follows from \autoref{thm:submain}.
\end{proof}

\begin{Ex}\label{ex:c5}\hfill
\begin{enumerate}[(i)]
\item Let $V:=\FF_5^3$, $d:=\diag(2,2,4)\in\GL(3,5)$ and
\[N:=V\rtimes\langle d\rangle=\mathtt{SmallGroup}(500,47).\] 
It can be shown that 
\[\Z_\infty(N)\kappa^{-1}(\Aut(N)')=V\langle d^2\rangle<N.\]
Let $D\cong C_4^3$ be the group of diagonal matrices in $\GL(3,5)$, and let $\sigma$ be the permutation matrix of a $3$-cycle. Then $N$ is subnormal in
 \[G:=V\rtimes(D\rtimes\langle\sigma
\rangle),\]
and another computation shows that $N=N\cap G'= R(N)$. Hence, \autoref{thm:subnecessary} can be strict.

\item Although $\Z_\infty(N)$ is nilpotent, the Fitting subgroup $\F(N)$ is not always contained in $R(N)$. For $N=\mathtt{SmallGroup}(768,1083540)$ one can show that $\F(N)=\pcore_2(N)$ has index $3$ while $R(N)=\pcore^2(N)$ has index $2$. We omit the details.
\end{enumerate}
\end{Ex}

Now we prepare to construct a family of examples where $R(N)<N$. We first record how a subnormal subgroup interacts with coprime action on $\pcore_p(G)$.

\begin{Lem}\label{lem:subcoprime}
Let $U$ be a subnormal subgroup of a group $G$, let $p$ be a prime and let $S\le U$ be a $p'$-subgroup. Then $[\pcore_p(G),S]=[U\cap\pcore_p(G),S]$.
\end{Lem}

\begin{proof}
Set $P:=\pcore_p(G)$, $J_0:=UP$ and $J_{i+1}:=U^{J_i}$ for $i\ge0$. Since $U$ is subnormal in $J_0$, we have $J_n=U$ for some $n\ge0$. Moreover $U\le J_{i+1}\le J_i$, so $Q_i:=J_i\cap P$ satisfies $Q_0=P$ and $Q_n=U\cap P$. From $[Q_i,S]\le[J_i,U]\le U^{J_i}=J_{i+1}$ and $[Q_i,S]\le P$ we obtain $[Q_i,S]\le Q_{i+1}$. 
The coprime action of $S$ on $Q_i$ yields
\[[Q_i,S]=[[Q_i,S],S]\le[Q_{i+1},S]\le[Q_i,S]\]
by \cite[8.2.7]{Kurzweil}.
Hence $[Q_i,S]=[Q_{i+1},S]$ for all $i$, and $[P,S]=[Q_0,S]=[Q_n,S]=[U\cap P,S]$.
\end{proof}

\begin{Thm}\label{thm:subhall}
Suppose that $N'$ is a cyclic Hall subgroup of $N$. Then $R(N)=\C_N(N')$. 
\end{Thm}

\begin{proof}
Set $D:=N'$ and choose a complement $S$ of $D$ in $N$, which exists by the Schur--Zassenhaus theorem. Since $D$ is abelian and $S\cong N/D$ is abelian, the group $N/[D,S]$ is generated by the central subgroup $D/[D,S]$ and the abelian subgroup $S[D,S]/[D,S]$, hence abelian. This gives $D=N'\le[D,S]$, and $D=\C_D(S)\times[D,S]$ by coprime action \cite[8.4.2]{Kurzweil}, so that $\C_D(S)=1$. Dedekind's lemma yields $\C_N(D)=D\C_S(D)$, and if $z=xy\in\Z(N)$ with $x\in D$ and $y\in\C_S(D)$, then $x\in\C_D(S)=1$. Therefore
\begin{equation}\label{eq:centralizer}
\Z(N)=\C_S(D),\qquad\C_N(D)=D\times\Z(N).
\end{equation}

Restriction to the characteristic subgroup $D$ is a homomorphism from $\Aut(N)$ into the abelian group $\Aut(D)$, so $\Aut(N)'$ centralizes $D$ and $\Aut(N)'\cap\Inn(N)\le\kappa(\C_N(D))$. By \eqref{eq:centralizer} we have $\C_N(D)=D\times\Z(N)=\kappa^{-1}(\kappa(D))$, and $\kappa(D)=\Inn(N)'\le\Aut(N)'$. Hence $\kappa^{-1}(\Aut(N)')=\C_N(D)$, and \autoref{thm:subnecessary} gives $\C_N(D)\le R(N)$.

For the reverse inclusion let $G$ be a group containing $N$ as a subnormal subgroup. Since $D=N'\le G'$, Dedekind's lemma gives $N\cap G'=D(S\cap G')$. By \eqref{eq:centralizer} it therefore suffices to show that $S\cap G'\le\C_S(D_p)$ for every prime divisor $p$ of $|D|$, where $D_p\in\Syl_p(D)$.

As $D_p$ is characteristic in $D$ and $D$ is characteristic in $N$, the group $D_p$ is a subnormal $p$-subgroup of $G$ and therefore $D_p\le P:=\pcore_p(G)$. Since $|N:D|$ is coprime to $p$, we have $D_p\in\Syl_p(N)$, and $N\cap P$ is a normal $p$-subgroup of $N$, whence $N\cap P=D_p$. Moreover $D_p=\C_{D_p}(S)\times[D_p,S]$ by \cite[8.4.2]{Kurzweil} and $\C_{D_p}(S)\le\C_D(S)=1$, so $[D_p,S]=D_p$. Thus \autoref{lem:subcoprime} gives
\[[P,S]=[D_p,S]=D_p.\]

Set $V:=P/\Phi(P)$. As $\Phi(P)$ is characteristic in $P\unlhd G$, conjugation induces a representation $\rho\colon G\to\GL(V)$ over $\FF_p$. Here, $[V,S]$ is the image $\overline{D_p}$ of $[P,S]=D_p$ in $V$. Since $S$ acts coprimely on $P$ and $[P,S]\ne1$, \cite[8.2.9]{Kurzweil} shows that $\overline{D_p}\ne1$. As $\overline{D_p}$ is a non-trivial quotient of the cyclic group $D_p$, we have $\overline{D_p}\cong D_p/\Phi(D_p)$ as $\FF_pS$-modules. Now $V=\C_V(S)\oplus[V,S]$ by \cite[8.4.2]{Kurzweil}, so $\det\rho(y)$ is the scalar by which $y\in S$ acts on $D_p/\Phi(D_p)$. The kernel of the natural map $\Aut(D_p)\to\Aut(D_p/\Phi(D_p))$ is a $p$-group by \cite[8.2.9]{Kurzweil} and $p\nmid|S|$, so $\det\rho(y)=1$ if and only if $y\in\C_S(D_p)$. Since $\FF_p^\times$ is abelian, $G'\le\Ker(\det\circ\,\rho)$ and therefore $S\cap G'\le\C_S(D_p)$.
\end{proof}

\begin{Ex}\label{ex:subQ8}
If all Sylow subgroups of $N$ are cyclic, then $N'$ is cyclic of order coprime to $|N:N'|$ by a theorem of Zassenhaus \cite[Satz~IV.2.11]{Huppert}, so \autoref{thm:subhall} applies to every such $N$. For $N\cong D_{2n}$ with $n$ odd, we obtain $R(N)=\C_N(N')\cong C_n$. Hence, the non-trivial linear character of $N$ is subnormally extendible. In particular, $R(S_3)=A_3$. This implies an observation by P.~Hall: $S_3$ is not subnormal in a perfect group (see \cite[Theorem~3]{Dark}). 
\end{Ex}

The invariant $R(N)$ can also be estimated by passing to quotients. Recall that a \emph{Fitting class} is a non-empty class $\MF$ of finite groups, closed under isomorphisms, which contains with every group $G$ all (sub)normal subgroups of $G$, and with any two normal subgroups $N_1,N_2\unlhd G$ in $\MF$ also their product $N_1N_2$. 
Every finite group $G$ then has a largest normal $\MF$-subgroup, the \emph{$\MF$-radical} $G_{\MF}$. For instance, the $\pi$-groups for a set of primes $\pi$, the nilpotent groups and the solvable groups form Fitting classes. The corresponding radicals are $\pcore_\pi(G)$, the Fitting subgroup $\F(G)$ and the solvable radical of $G$.
Since $N_{\MF}$ is characteristic in $N$, an easy induction on the length of a subnormal chain shows that $N_{\MF}\le G_{\MF}$ whenever $N\sub G$. 

\begin{Prop}\label{thm:reduction}
For every Fitting class $\MF$, we have $R(N)N_{\MF}/N_{\MF}\le R(N/N_{\MF})$.
\end{Prop}

\begin{proof}
Let $N\sub G$ and set $P:=G_{\MF}$. By the remark above, $N_{\MF}=N\cap P$. Hence $xN_{\MF}\mapsto xP$ is an isomorphism from $N/N_{\MF}$ onto $NP/P$, and $NP/P$ is subnormal in $G/P$, being the image of $N$ under the natural map.

Since $R$ is compatible with isomorphisms, this isomorphism maps $R(N/N_{\MF})$ onto $R(NP/P)$. For $x\in N\cap G'$ we have $xP\in NP/P\cap(G/P)'\le R(NP/P)$, so that $xN_{\MF}\in R(N/N_{\MF})$. Hence $(N\cap G')N_{\MF}/N_{\MF}\le R(N/N_{\MF})$. Since $G$ was arbitrary, the claim follows.
\end{proof}

\begin{Ex}\label{ex:solvable}\hfill
\begin{enumerate}[(i)]
\item\label{item:s4} For $N=S_4$ we have $\pcore_2(N)=V\cong C_2^2$, $N/V\cong S_3$ and $R(S_3)=A_3$ by \autoref{ex:subQ8}. Since the preimage of $A_3$ in $S_4$ is $A_4$, \autoref{thm:reduction} gives $R(S_4)\le A_4=S_4'\le R(S_4)$. Hence $R(S_4)=A_4$.

\item\label{item:complete} Let $N$ be an indecomposable solvable complete group. By \cite[Corollary~B and p.~436]{HeinekenLennox}, every group $G\rsub N$ satisfies $N\cap G'\le N'\F(N)$. Therefore, $R(N)\le N'\F(N)$. This generalizes \ref{item:s4}.
\end{enumerate}
\end{Ex}

In the examples so far $N$ was solvable. For almost simple groups $N$ the invariant $R(N)$ can be computed inside $\Aut(S)$, where $S$ is the socle of $N$.

\begin{Thm}\label{thm:almostsimple}
Let $N$ be almost simple with socle $S$ and regard $N$ as a subgroup of $\Aut(S)$ via conjugation. Then 
\[R(N)=\bigl\langle N\cap B':N\sub B\le\Aut(S)\bigr\rangle.\]
In particular, $R(N)=S$ if $\Out(S)$ is abelian, and $R(N)=N\cap\Aut(S)'$ if $N\sub\Aut(S)$.
\end{Thm}

\begin{proof}
The right-hand side is contained in $R(N)$ by definition. Conversely, let $N\sub G$. We construct a group $B$ with $N\sub B\le\Aut(S)$ and $N\cap G'\le N\cap B'$.

Since $S$ is a subnormal non-abelian simple subgroup of $G$, every conjugate $S^g$ with $g\in G$ is a component of $G$. By \cite[6.5.2]{Kurzweil}, $S^g\le N$ or $[S^g,N]=1$. In the first case $S^g$ and $S$ are components of $N$, so $S^g=S$ or $[S^g,S]=1$ by \cite[6.5.3]{Kurzweil}, and the latter would give $S^g\le\C_N(S)=1$. Hence, $N$ normalizes every conjugate of $S$ and centralizes every conjugate different from $S$.

Let $G_1:=\N_G(S)$ and let $\rho\colon G_1\to\Aut(S)$ be the conjugation map, so that $\rho$ restricts to the given embedding of $N\le G_1$. Set $B:=\rho(G_1)$. As $N$ is subnormal in $G_1$, its image $N$ is subnormal in $B\le\Aut(S)$. Let $1=g_1,g_2,\ldots,g_r$ be a right transversal of $G_1$ in $G$, so that $S^{g_1},\ldots,S^{g_r}$ are the distinct conjugates of $S$. Let $x\in N$. Since $x$ normalizes $S^{g_i}$, we have $g_ixg_i^{-1}\in G_1$, i.\,e. $x$ fixes every coset $G_1g_i$ under right multiplication. Hence the transfer homomorphism $\tau\colon G\to G_1/G_1'$ satisfies
\[\tau(x)=\prod_{i=1}^rg_ixg_i^{-1}G_1'\]
(see \cite[7.1.1]{Kurzweil} and formula (5) there). For $i\ge2$ the element $x$ centralizes $S^{g_i}$, so $g_ixg_i^{-1}$ centralizes $S$ and $\rho(g_ixg_i^{-1})=1$. Since $\rho(G_1')=B'$, the map $\rho$ induces a homomorphism $\overline\rho\colon G_1/G_1'\to B/B'$ with $\overline\rho(\tau(x))=\rho(x)B'$. If $x\in N\cap G'$, then $\tau(x)=1$, because $G_1/G_1'$ is abelian, and consequently $\rho(x)\in B'$. This shows $N\cap G'\le N\cap B'$.

If $\Out(S)$ is abelian, then every $B$ with $N\sub B\le\Aut(S)$ satisfies $S=S'\le B'\le S$, since $B/S$ is abelian. Hence $R(N)=S$ in this case. If $N\sub\Aut(S)$, then $B=\Aut(S)$ is admissible, and $N\cap B'\le N\cap\Aut(S)'$ for every admissible $B$.
\end{proof}

\begin{Cor}\label{cor:symmetric}
For $n\ge3$ we have $R(S_n)=A_n$. In particular, the sign character is subnormally extendible.
\end{Cor}

\begin{proof}
For $n\in\{3,4\}$ this is \autoref{ex:subQ8} and \autoref{ex:solvable}\ref{item:s4}. For $n\ge5$, $S_n$ is almost simple with socle $A_n$. Since $\Out(A_n)$ is abelian (of order at most $4$), $R(S_n)=A_n$ by \autoref{thm:almostsimple}. The second claim follows from \autoref{thm:submain}.
\end{proof}

\section*{Acknowledgment}
Anthropic Claude Opus 5.5 was used for the preparation of this paper. I have verified all statements and proofs, and take full responsibility for the content.


\begin{thebibliography}{1}

\bibitem{Dark}
R.~S. Dark, \textit{On subnormal embedding theorems for groups}, J. London
  Math. Soc. \textbf{43} (1968), 387--390.

\bibitem{GAP4161}
The GAP~Group, \textit{{GAP -- Groups, Algorithms, and Programming, Version
  4.16.1}}; 2026, \url{https://www.gap-system.org}.

\bibitem{HeinekenLennox}
H. Heineken and J.~C. Lennox, \textit{The subnormal embedding of complete
  groups}, J. Algebra \textbf{90} (1984), 435--445.

\bibitem{Huppert}
B. Huppert, \textit{Endliche {G}ruppen. {I}}, Grundlehren der Mathematischen
  Wissenschaften, Vol. 134, Springer-Verlag, Berlin, 1967.

\bibitem{Isaacs}
I.~M. Isaacs, \textit{Character theory of finite groups}, AMS Chelsea
  Publishing, Providence, RI, 2006.

\bibitem{Karpilovsky}
G. Karpilovsky, \textit{The {S}chur multiplier}, London Mathematical Society
  Monographs. New Series, Vol. 2, The Clarendon Press Oxford University Press,
  New York, 1987.

\bibitem{Kurzweil}
H. Kurzweil and B. Stellmacher, \textit{The theory of finite groups},
  Universitext, Springer-Verlag, New York, 2004.

\bibitem{Navarro}
G. Navarro, \textit{Characters and blocks of finite groups}, London
  Mathematical Society Lecture Note Series, Vol. 250, Cambridge University
  Press, Cambridge, 1998.

\end{thebibliography}
\end{document}